\documentclass[12pt]{amsart}

\usepackage{amsmath,amsfonts,amssymb, bm}

\usepackage{a4wide,comment,amscd}

\newtheorem{lemma}{Lemma}[section]

\newtheorem{prop}[lemma]{Proposition}
\newtheorem{thm}[lemma]{Theorem}
\newtheorem{cor}[lemma]{Corollary}
\theoremstyle{definition}
\newtheorem{defn}[lemma]{Definition}

\theoremstyle{remark}
\newtheorem{rmk}[lemma]{Remark}

\newcommand{\J}{\mathcal{J}}
\newcommand{\G}{\mathcal{G}}
\newcommand{\E}{\mathcal{E}}
\newcommand{\C}{\mathcal{C}}
\newcommand{\D}{\mathcal{D}}
\renewcommand{\H}{\mathcal{H}}
\newcommand{\A}{\mathcal{A}}
\renewcommand{\L}{\mathcal{L}}
\newcommand{\BbZ}{\mathbb{Z}}

\newcommand{\BbN}{\mathbb{N}}

\renewcommand{\epsilon}{\varepsilon}
\newcommand{\be}{\beta}

\numberwithin{equation}{section}
\numberwithin{table}{section}

\begin{document}

\title[Open maps: small and large holes]{Open maps: small and large holes \\ with unusual properties}
\author{Kevin G. Hare and Nikita Sidorov}
\address{}
\address{Department of Pure Mathematics \\
University of Waterloo \\
Waterloo, Ontario \\
Canada N2L 3G1}
\thanks{Research of K. G. Hare was supported by NSERC Grant RGPIN-2014-03154.}
\email{kghare@uwaterloo.ca}
\address{
School of Mathematics, The University of Manchester,
Oxford Road, Manchester M13 9PL, United Kingdom.}
\email{sidorov@manchester.ac.uk}

\date{\today}
\begin{abstract}
Let $X$ be a two-sided subshift on a finite alphabet endowed
with a mixing probability measure which is positive
on all cylinders in $X$. We show that there exist arbitrarily
small finite overlapping union of shifted cylinders which intersect every orbit under the shift map.

We also show that for any proper subshift $Y$ of $X$ there exists a finite overlapping unions
of shifted cylinders such that its survivor set contains $Y$ (in particular, it can have entropy arbitrarily
close to the entropy of $X$). Both results may be seen as somewhat counter-intuitive.

Finally, we apply these results to a certain class of hyperbolic algebraic automorphisms of a torus.
\end{abstract}

\maketitle

\section{Introduction}
\label{sec:intro}

This paper is concerned with an area of dynamics which is usually referred to as ``maps with holes'', ``open maps'' or ``open dynamical systems''.
Let $X$ be a compact (or pre-compact) metric space and $f:X\to X$ be a map with positive topological entropy.
Let $H\subset X$ be an open set which we regard as a {\em hole}.

In this area of dynamics, we consider the set of points that do not fall into a ``hole'' under iterations of a map.
This is a well studied area with numerous papers.
The first paper in this area was \cite{PY1979}, where they considered a game of billiards with a hole somewhere in the table.
The rate at which a random billiard ball ``escaped'' was considered, and depended upon both the size and the location of the hole.
It is clear that if a hole is enlarged, then the rate of escape will not decrease, and often increase.
In \cite{AB2010, BY2011}, it was shown that the location could sometimes play a more significant role in the escape rate than the size
of the hole.
An important factor in the escape rate is the set of periodic points of the map that fall into the hole \cite{BJP2014}.
For a good history of these problems, see \cite{Demers2006, Demers2014, FP}.

The problem that we study, while related, is not exactly the same.
We are concerned with the set of points that will always avoid a hole,
    that is, that will never escape. The precise connection between the theory of escape rates and our results is
    explained in detail in Remark~\ref{rem:final} at the end of the paper. 

Formally, we denote by $\J(H)$ the set of all points in $X$ whose $f$-orbit does not intersect $H$ and call it
{\em the survivor set}. Clearly, a survivor set is $f$-invariant, and in a number of recent papers certain
dynamical properties of the map $f|_{\J(H)}$ have been studied -- see, e.g., \cite{Rafa} and
references therein.

It seems that a more immediate issue here is the ``size'' of a survivor set.
At first sight, it would seem plausible that if $H$ is ``large'', then $\J(H)$ is countable -- or even empty.
On the other hand, if it is ``small'', then one might expect the Hausdorff dimension of $\J(H)$ to be positive.

The starting point for this line of research has been the case when
$X=[0,1]$ and $T(x)=2x\bmod1$, the doubling map with $T(1)=1$.
Assume our hole to be connected, so we have $H=(a,b)\subsetneq(0,1)$.
We denote $\J(H)$ by $\J(a,b)$.

\begin{thm}[\cite{GS}]
\begin{enumerate}
\item
We always have $\dim_H \J(a,b)>0$ if
$$
b-a<\frac14\cdot\prod_{n=1}^{\infty}\left(1-2^{-2^n}\right) \approx0.175092,
$$
and this bound is sharp.
\item If $b-a>\frac12$, then $\J(a,b)=\{0,1\}$.\footnote{If we regard the doubling map
as the map $z\mapsto z^2$ on the unit circle $S^1$, then $\frac12$ needs to
be replaced with $\frac23$, in which case $\J(a,b)\subset\{1\}$.}
\end{enumerate}
\end{thm}

A similar claim holds for the $\beta$-transformation $x\mapsto\beta x\bmod1$
for $\beta\in(1,2)$ -- see \cite{Lyn}. Thus, in the one-dimensional setting one's
naive expectations prove to be spot on.

The situation is however very different for the baker's map. Namely, put $X=[0,1]^2$;
the baker's map $B:X\to X$ is the natural extension of the doubling map, conjugate to the
shift map on the set of bi-infinite sequences. We have
 \begin{equation*}
  B(x,y) = \begin{cases}
            (2x, \frac{y}{2}) &\text{ if } 0 \leq x < \frac{1}{2}, \\
            (2x-1, \frac{y+1}{2}) &\text{ if } \frac{1}{2} \leq x < 1.
           \end{cases}
 \end{equation*}

We will say that an open set $H$ is a {\em complete trap} for the baker's map
if $\J(H)$ does not contain any points except, possibly, an orbit on the boundary
of $X$. (In which case it is symbolically $\dots 0000\dots$, $\dots 1111\dots$,
    $\dots 11110000\dots$ or $\dots 00001111\dots$)

\begin{thm}\cite[Theorem~2.2]{CHS}
\label{thm:CHS}
For any $\epsilon > 0$ exists a connected complete trap $H$
    such that its Lebesgue measure Leb is less than $\epsilon$.
\end{thm}

By comparison, we also know

\begin{thm}\cite[Theorem~2.4]{CHS}\label{thm:CHS-large}
\label{thm:CHS2}
If $H$ is a hole whose closure is disjoint from the boundary of the square,
    then $\dim_H\J(H) > 0$.
In particular, for any $\epsilon>0$ there exists
    $H$ such that $\text{Leb}(\J(H))>1-\epsilon$.
\end{thm}

The purpose of this note is to extend these two results from the full shift on two symbols to more
general subshifts (Theorems~\ref{thm:main} and \ref{thm:large-holes 2})
and apply these to the generalized Pisot toral automorphisms (Section~\ref{sec:toral}).

\section{Symbolic model: small holes}
\label{sec:symbolic}

Let $\A$ be a finite alphabet.
Let $\sigma:\A^\BbN\to \A^\BbN$ denote the left shift, i.e.,
\[
\sigma(\dots x_{-2} x_{-1} x_0.\ x_1 x_2 x_3\dots) =
                \dots x_{-2} x_{-1} x_0 x_1 .\ x_2 x_3\dots
                \]
We define distance on $\A^\BbN$ with
    $\A=\{0,1,\dots,k-1\}$ in the usual way as
    \[ d(\bm x, \bm y)=\sum_{i\in\mathbb Z} \frac{|x_i-y_i|}{k^{|i|}}.\]
This topology is equivalent to the product topology.
We say that $X \subset \A^\BbN$ is a {\em subshift} if $X$ is closed
    under this topology, and $\sigma(X) = X$.
Suppose $\mu$ is a probability $\sigma$-invariant
    measure on $X$ which we assume to be positive on all cylinders.
We denote the set of admissible words of length~$n$ for $X$ by $\L_n(X)$
and put $\L(X)=\bigcup_{n\ge 0} \L_n(X)$, the {\em language of $X$}.

The topological entropy of a subshift $X$ is defined by the formula
\[
h(X)=\lim_{n\to\infty} \frac{\log \#\L_n(X)}n.
\]
Recall that a subshift $(X,\sigma)$ is called {\em irreducible} if, for every ordered pair of
words $u$ and $v\in\L(X)$, there is a $w\in\L(X)$ with $uwv\in\L(X)$.
We say a subshift is {\em sofic} if there exists a regular language (i.e.,
a language accepted by a finite automaton) such that $X$ is the set
of all infinite sequences that do not contain a subword from this regular language.
For more detailed see \cite[Chapter~3]{LindMarcus}.

An irreducible sofic subshift is known to have a unique
measure of maximal entropy $\mu$ (see \cite{Weiss}), which has the following property (see
\cite[Lemma~4.8]{Hu-Lin}):
there exists $\theta\in(0,1)$ such that
\begin{equation}\label{eq:sofic}
\mu[w_{-n+1}\dots w_0.w_1\dots w_n]\asymp \theta^n,\quad \forall w_{-n+1}\dots w_n\in\L_{2n}(X).
\end{equation}
(Clearly, $\theta=\frac12\exp(-h(X))$,
in view of the Shannon--McMillan--Breiman theorem.)

Let $w = w_1 w_2 \dots w_n \in\L(X)$; we denote by $[w]$ the cylinder given by
    all $\{x_j\}_{j\in \BbZ} \in X$ with
    $x_{1} = w_1, x_{2} = w_2, \dots, x_{n} = w_n$.

Recall that $\mu$ is called {\em non-atomic} if $\mu(\{\bm x\})=0$ for any $\bm x\in X$.
We say that a measure $\mu$ is {\em mixing} if for any pair of cylinders $[w], [w']$ we have
$\mu(\sigma^{-n}[w]\cap [w'])\to\mu([w])\mu([w'])$ as $n\to+\infty$.
Similarly to the case of the baker's map, we say that a Borel $A\subset X$ is a
{\em complete trap} if it intersects all
orbits in $X$, i.e., for any $x\in X$ there exists $n\in\BbZ$ such that $\sigma^n x\in A$.

\begin{rmk}\label{rmk:non-atomic}
Notice that a mixing measure positive on all cylinders is always non-atomic.
Indeed, assume that $\mu(A)>0$, where $A=\{\bm x\}$. Then by mixing,
$\mu(\sigma^{-n}A\cap A)\to\mu(A)^2$, which is impossible unless $\bm x$ is a fixed point
for $\sigma$. If $\bm x$ is such, then it must have $x_i\equiv j$ for all $i\in\mathbb Z$
and we have $\mu(A)=1$, which is impossible, since $\mu$ cannot be supported by
a single orbit, being positive on all cylinders.
\end{rmk}

\begin{rmk}\label{rmk:un}
A similar notion exists in Combinatorics on Words, namely, unavoidable sets.
A subset $S$ of $\mathcal A^*$ is called {\em unavoidable} if any word in $\mathcal A^*$
contains an element of $S$ as a factor. These have been introduced by M.-P.~Sch\"utzenberger in \cite{Schutz},
and various results, mostly regarding their size, have been proved since -- see, e.g.,
\cite{Champ} and references therein.
\end{rmk}

In the case of $[0,1]$ and $[0,1]^2$ in Section \ref{sec:intro}, we found connected components with
    unusual properties.
If $X$ is a subshift, then $X$ is totally disconnected, so $X$ has no
    non-trivial connected components.
Instead we consider what we believe is a natural analog of connected, in the
    case of subshifts.
That is, being a finite union of overlapping shifted
    cylinders.
More formally
\begin{defn}
Let $C = \{\sigma^{-r_1}[w_1], \dots, \sigma^{-r_n}[w_n]\}$ be a finite
    collection of shifted cylinders.
We say that $C$ is {\em overlapping}
    if for all $i$ and $j$ there exists $i = i_1, i_2, \dots, i_n = j$
    such that $\mu( \sigma^{-i_k}[w_{i_k}] \cap
    \sigma^{-i_{k+1}}[w_{i_{k+1}}]) > 0$.
\end{defn}

The main result of this section is
    that for each two-sided subshift with mild assumptions
    there exist arbitrarily small overlapping complete traps.
More precisely:

\begin{thm}
\label{thm:main}
Let $X$ be a two-sided subshift on a finite alphabet $\A$ endowed with
    a shift-invariant probability measure $\mu$ on $X$ which we assume to be
    positive on all cylinders and mixing (hence non-atomic) in $X$.

Then for any $\epsilon>0$ there exists a finite union of shifted cylinders
    $\C := \bigcup_{i=1}^n \sigma^{-r_i}[w_i]$ such that
\begin{enumerate}
\item $\mu(\C)<\epsilon$.
\item $\C$ is a complete trap.
\item for all $i, j$ we have
    $\mu(\sigma^{-r_i}[w_i] \cap \sigma^{-r_j}[w_j]) > 0$.
\end{enumerate}
\end{thm}

This will be proven by a series of lemmas.

\begin{lemma}\label{lem:non-atomic}
Assume that $\mu$ is a shift-invariant non-atomic measure on $X$ and $w=w_1w_2\dots$
is such that $w_1\dots w_n\in\L_n(X)$ for all $n\ge1$. Then $\mu[w_1\dots w_n]\to0$
as $n\to\infty$.
\end{lemma}

\begin{proof}Since $\mu$ is shift-invariant, we have
$\mu[.w_1\dots w_n]=\mu[w_1\dots w_{\lfloor n/2\rfloor}.w_{\lfloor n/2\rfloor+1}\dots w_n]$, so
the claim follows from $\mu$ being non-atomic.
\end{proof}

\begin{lemma}\label{lem:partition-cylinder}
Let $\mu$ is a shift-invariant non-atomic measure on a subshift $X$.
Let $\C = \bigcup \sigma^{-r_i}[w_i]$ be a finite union of shifted cylinders with all $r_i \geq 0$.
For all $\epsilon > 0$
    we can write $\C = \bigcup [w_i']$ as a disjoint union of cylinders where $w_i'$
    have equal length, and $\mu([w_i']) < \epsilon$ for all $i$.
\end{lemma}

\begin{proof}
It is worth observing that the $\sigma^{-r_i}[w_i]$ may overlap, and the $w_i$ may have different
    lengths.
For $N \geq |w_i|+r_i$, we see that we can find a prefix $u$ of length $r_i$ and suffix $v$ of length
    $N - r_i - |w_i|$ such that $u w_i v \in \L(X)$ is a word of length $N$.
We let $S$ be the set of all such $u w_i v$ for all pairs $w_i$ and $r_i$.
This gives us a set $\C = \bigcup_{w_i' \in S} [w_i']$ of disjoint union of cylinders of equal length.
The fact that we can choose $N$ sufficiently large so that $\mu([w]) < \epsilon$ follows from
    Lemma \ref{lem:non-atomic}.
\end{proof}

\begin{lemma}\label{lem:cylinder-overlap}
Let $X$ be a two-sided subshift on a finite alphabet $\A$ endowed with
    a shift-invariant probability measure $\mu$ on $X$ which we assume to be
    positive on all cylinders and mixing (hence non-atomic) in $X$.
Let $\C = \bigcup [w_i]$ be a finite union of disjoint cylinders and a complete trap.
Then there exists $r_i \geq 0$ such that such that
\begin{enumerate}
\item     $\C' = \bigcup \sigma^{-r_i} [w_i]$ is a complete trap,
\item $\mu(\C') \leq \mu(\C)$ and
\item for all $i, j$ we have $\mu(\sigma^{-r_i}[w_i] \cap \sigma^{-r_j}[w_j]) > 0$.
\end{enumerate}
\end{lemma}

\begin{proof}
As $\mu$ is positive on all cylinders, we see that
    $\mu([w_i]) > 0$ for all choices $i$.
For all $w_i$ and $w_j$ there will exist a $M_{i,j}$ such that for all $n \geq M_{i,j}$ we have
    $\mu([w_i]) \cap \sigma^{-n}([w_j])) > 0$.
Let $M := \max_{i,j} M_{i,j}$.
Taking $r_i = M \cdot i$ satisfies the desired properties.
\end{proof}

\begin{lemma}\label{lem:simple}
Let $n\ge2,\ 0\le x_1\le\dots \le x_n\le\frac12$ with $\sum_{i=1}^n x_i=1$.
Then there exists $k\in\{1,\dots,n\}$ such that
$\sum_{i=1}^k x_i\ge \frac14$ and $\sum_{i=k+1}^n x_i \ge\frac14$.
\end{lemma}
\begin{proof}If $x_1\ge\frac14$, then we take $k=1$. Otherwise
let $k$ be such that $\sum_1^{k-1} x_i<\frac14$ and $\sum_1^k x_i\ge \frac14$.
Then $\sum_{k+1}^n x_i = 1-\sum_1^{k-1} x_i - x_k\ge \frac12-\sum_1^{k-1} x_i\ge \frac14$.
\end{proof}

\begin{lemma}\label{lem:main}
Let $X$ be a two-sided subshift on a finite alphabet $\A$ endowed with
    a shift-invariant probability measure $\mu$ on $X$ which we assume to be
    positive on all cylinders and mixing (hence non-atomic) in $X$.

Then for any $\epsilon>0$ there exists a
    $\C := \bigcup [w_i]$, a finite union of disjoint cylinders
    such that
\begin{enumerate}
\item $\mu(\C)<\epsilon$.
\label{en:1}
\item $\C$ is a complete trap.
\label{en:2}
\end{enumerate}
\end{lemma}

\begin{proof}
Put $\C_0= \bigcup_{a \in \A}  [a]$.
Clearly $\C_0 = X$, and hence is a complete trap.

We proceed by induction.
Assume we have a finite collection of cylinders
    $\C_n = \bigcup \sigma^{-r_i} [w_i]$ which is a complete trap.
Using Lemma~\ref{lem:partition-cylinder} we write
    $\C_n = \bigcup [w_i']$, a disjoint union of cylinders
    where $\mu([w_i']) < \mu(\C_n)/2$ and all $w_i'$ are the same length,
    say $N$.

By Lemma~\ref{lem:simple} and the fact that all cylinders have
    $\mu([w']) < \mu(\C_n)/2$, we can
    partition the set of cylinders into two sets $\C_n'$ and $\C_n''$ where
\[
\mu(\C_n) / 4 < \min\,\Bigl\{\mu\Bigl(\bigcup_{w'\in \C_n'}[w']\Bigr),
\mu\Bigl(\bigcup_{w''\in \C_n''}[w'']\Bigr)\Bigr\}.
\]

Since $\mu$ is mixing, there exists $\ell$ such that
$\mu(\sigma^{-\ell}[w']\cap [w''])\ge\frac12\mu([w'])\mu([w''])$ for all
$[w']\in\C_n'$ and all $[w'']\in\C_n''$. Put
\[
\C_{n+1} = \left(\bigcup_{w' \in \C_n'} [w']\right) \bigcup
           \left(\bigcup_{w'' \in \C_n''} \sigma^{-\ell}[w'']\right)
\]
We have
\begin{align*}
\mu(\C_{n+1})&= \sum_{w' \in \C_n'} \mu([w']) +
    \sum_{w'' \in \C_n''} \mu([w''])
    -\sum_{w'\in\C_n', w''\in\C_n''}
    \mu([w'] \cap \sigma^{-\ell} [w'']) \\
&\leq \mu(\C_n') + \mu(\C_n'')
    -\sum_{w'\in\C_n'} \sum_{w''\in\C_n''}
     \frac{1}{2} \mu([w'])\mu[w'']) \\
&= \mu(\C_n)  - \frac{1}{2} \mu(\C_n') \mu(\C_n'')  \\
&\leq \mu(\C_n)  - \frac{1}{32} \mu(\C_n)^2
\end{align*}
Put $t_n=\mu(\C_n)$. Then we have
\begin{equation}\label{eq:tn}
t_{n+1}\le t_n - \frac{t_n^2}{32}.
\end{equation}
Clearly, $t_n$ is decreasing and positive; let $L=\lim_{n\to\infty}t_n$. Then $L\le L-L^2/32$,
whence $L=0$.

Choose $n$ such that $\mu(\C_n) < \epsilon$ and use Lemma \ref{lem:partition-cylinder} to write
$\C := \C_n$.  This union in the desired form.
\end{proof}

\begin{rmk}
It follows from the proof of Lemma \ref{lem:main} that $\mu$ being mixing can
be replaced with the following, weaker, condition: there exists a $\delta > 0$
such that for all cylinders $[w], [w'] \subset X$ we have
$\liminf_{n \to \infty} \mu(\sigma^{-n}[w] \cap [w']) \ge \delta \mu[w] \mu[w']$ --
provided we assume that $\mu(\{\bm x\})=0$ for any fixed point $\bm x \in X$.
\end{rmk}

\begin{rmk}
The sequence $t_n$ in (\ref{eq:tn}) tends to 0 as $\approx 1/n$.  This is consistent with the theory
of unavoidable sets (see Remark~\ref{rmk:un}), where it is shown for that the minimal size of
an unavoidable set (for the full shift on $\mathcal A$) is $\gg |\mathcal A|^n/n$ (\cite{Schutz, Mykk}).
\end{rmk}

\section{Symbolic model: large holes} \label{sec:large}
The goal of this section is to extend Theorem~\ref{thm:CHS-large} to more general subshifts.

\begin{thm}\label{thm:large-holes 2}
Let $X$ be a subshift endowed with a mixing probability measure
of maximal entropy $\mu$.
Let $Y \subset X$ be a subshift such that
    $ 0 < h(Y) < h(X)$ (i.e., a proper subshift of $X$).

Then for any $\epsilon>0$ there exists a finite overlapping union of cylinders
$\G$ such that $\mu(\G)>1-\epsilon$ and $Y\subset \J(\G)$. (So,
in particular, $h(\J(\G))>0$.)
\end{thm}

\begin{proof}
Let $\L_n(Y)$ be as above. Since $h(Y) < h(X)$, we have
    $\frac{\# \L_n(Y)}{\# \L_n(X)} \to 0$ as $n \to \infty$.

Let $\Sigma_n = \{[w]: w \in \L_n(X)\}$ and
    $\Sigma_n' = \Sigma_n \setminus \{[w]: w \in \L_n(Y)\}$.
We have
\begin{align*}
\mu(\Sigma_n')
    & = \mu([w]: w \in \L_n(X)) - \mu([w]: w \in \L_n(Y)) \\
    & = 1 - \mu([w]: w \in \L_n(Y)) \\
    & \to 1,
\end{align*}
as $n \to \infty$.

We see that $Y \subset \J(\Sigma_n')$ and thus, $h(\J(\Sigma_n')) > 0$.
To get overlapping, let $[w^*] \in \Sigma_n'$ be such that
    $\mu([w^*])$ is minimized.
We note here that $\mu([w^*]) > 0$ as $\mu$ is positive on all
    cylinders.
Consider \[ \G_n = \left( \Sigma_n' \setminus \{[w^*]\} \right)
    \cup \{ \sigma^{-K}([w^*]) \} \]
    for $K$ sufficiently large so that $\mu([w] \cap \sigma^{-K}([w^*])) > 0$
     for all $[w] \in \Sigma_n'$ (by mixing).
That is, so that $\G_n$ is overlapping.

We have $h(\J(\G_n)) > 0$ as $Y \subset \J(\G_n)$.
Let $S_n$ be the number of cylinders in $\Sigma_n'$.
Clearly $S_n \to \infty$ as $n \to \infty$.
We see that
    $\mu(\G_n) > \frac{S_n - 1} {S_n} \mu(\Sigma_n') \to 1$ as $n \to \infty$.
Furthermore, $\G_n$ is overlapping.
Taking $n$ sufficiently large, we get that $\G := \G_n$ has the desired property.
\end{proof}

\begin{rmk}
The condition $0<h(Y)<h(X)$ may be seen as a symbolic analogue of $\overline H$
being disjoint from the boundary of the square for the baker's map.
\end{rmk}

\section{Application: hyperbolic toral automorphisms}
\label{sec:toral}

\subsection{Background}
Let $\be>1$ be non-integer and $\tau_\be$ be the {\em $\be$-transformation}, i.e., the
map from $[0,1)$ onto itself acting by the formula
$$
\tau_\be(x)=\be x\bmod1.
$$
As is well known, to ``encode" it, one needs to apply the greedy
algorithm in order to obtain the digits in the $\be$-expansion, namely,
\begin{equation}\label{eq:beta}
x=\pi_\be((a_n)_1^\infty):=\sum_{n=1}^\infty a_n\be^{-n},
\end{equation}
where $a_n=\lfloor\be\tau_\be^{n-1}x\rfloor,\ n\ge1$. Then the one-sided
left shift $\sigma_\be^+$ on the space $X_\be^+$ of all possible sequences
$(a_n)_1^\infty$ which can be obtained this way, is isomorphic to
$\tau_\be$, with the conjugating map given by (\ref{eq:beta}).
We will call $\sigma_\be^+$ the {\em $\be$-shift}. It is obviously a
proper subshift of the full shift on
$\prod_1^\infty\{0,1,\dots,\lfloor\be\rfloor\}$.

W.~Parry in his seminal paper
\cite{Pa} proved the following. Let the
sequence $(d_n)_1^\infty$ be defined as follows: let
$1=\sum_{1}^{\infty}d_k' \be^{-k}$ be the greedy expansion of 1,
i.e, $d_n'=\lfloor\be\tau_\be^{n-1}1\rfloor,\ n\ge1$; if the tail of the
sequence $(d_n')$ differs from $0^\infty$, then we put
$d_n\equiv d_n'$. Otherwise let $k=\max\,\{j:d_j'>0\}$, and
$(d_1,d_2,\dots):= (d_1',\dots,d_{k-1}',d_k'-1)^\infty$.

Then
\[
X_\be^+=\{(a_n)_1^\infty\in\{0,1,\dots,\lfloor\beta\rfloor\}^{\mathbb N}
: a_na_{n+1}\dots \prec d_1d_2,\ n\ge1\},
\]
(where $\prec$ stands for the lexicographic order)
and the following diagram commutes:
\[
\begin{CD}
X_\be^+ @>{\sigma_\be^+}>> X_\be^+ \\
@V{\pi_\be}VV @VV{\pi_\be}V \\
[0,1) @>{\tau_\be}>> [0,1)
\end{CD}
\]
We restrict our attention to those $\beta$ where
     assume that $(d'_n)_1^\infty$ does not have unbounded strings of
    $0s$.
It follows that there exists $\ell=\ell(\be)\ge1$ such that
\begin{equation}\label{eq:ell}
u,v\in\L(X_\be^+)\implies u0^\ell v\in\L(X_\be^+).
\end{equation}
Similar to before, let $w = w_1 w_2 \dots w_n \in\L(X_\beta^+)$;
    we denote by $[w]^+$ the cylinder given by
    all $\{x_j\}_{j\in \BbN} \in X_\beta^+$ with
    $x_{1} = w_1, x_{2} = w_2, \dots, x_{n} = w_n$.
As also proved in \cite{Pa, Re}, there exists a unique probability measure
$\mu_\beta^+$ invariant under the $\tau_\beta$ which is equivalent to
the Lebesgue measure, with a density bounded from 0 and $\infty$.
Furthermore, $\mu_\beta(\pi_\beta([w]^+))\asymp \beta^{-n}$ for any
cylinder $[w]^+ = [w_1\dots w_n]^+ \subset X_\be^+$, provided
$(d_n')_1^\infty$ does not have unbounded strings of 0s.

Let $m\ge2$ and $M$ be an $m\times m$ matrix with integer entries and determinant $\pm1$.
Then $M$ determines the algebraic automorphism of the $m$-torus $\mathbb T^m:=\mathbb R^m/\mathbb Z^m$,
which we will denote by $T_M$. That is, $T_M\bm x=M\bm x\bmod\mathbb Z^m$.
We assume $M$ to have a characteristic polynomial $p$ irreducible over $\mathbb Q$.

Assume that $T_M$ is hyperbolic, i.e., that $p$ has no roots of modulus~1.
Let $\bm t$ be a {\em homoclinic point} for $T_M$, i.e., $T_M^n\bm t\to\bm0$ as $n\to\pm\infty$. Let $X$ be
a two-sided subshift on a finite alphabet, and define the map $\phi_{\bm t}:X\to\mathbb T^m$ as follows:
\begin{equation}\label{eq:phit}
\phi_{\bm t}(\bm a)=\sum_{n\in\mathbb Z} a_nT_M^{-n}\bm t,
\end{equation}
where $\bm a=(a_n)_{n\in\mathbb Z}$.
These maps have been studied in \cite{Sch, S01, SV1, SV2}.
Note first that $\phi_{\bm t}$ is well defined,
since, as is well known, $T_M^n\bm t\to\bm0$ at an exponential rate, so this bi-infinite series converges.
Also, $\phi_{\bm t}$ is H\"older continuous, for the same reason. Most
importantly, we have $\phi_{\bm t}\sigma=T_M\phi_{\bm t}$, i.e., $\phi_{\bm t}$ semiconjugates the shift and $T_M$.
It is known (see \cite{Sch}) that one can choose $L\ge1$ large enough so that if $X=\{-L,\dots, L\}^{\mathbb Z}$
is the full shift, then $\phi$ is surjective.

Assume now that $p$ has one real root of modulus greater than~1 ($\beta$, say) and the rest are less than~1 in modulus.
Then $\beta$ is called a {\em Pisot number} (a Pisot unit, to be more precise, in view of $\det M=\pm1$)
and $T_M$ a {\em Pisot automorphism}. In this case we have
a natural choice for $X$, namely, $X=X_\beta$, i.e., the natural extension of $X_\be^+$
endowed with the measure $\mu_\be$, the natural
extension of $\pi_\beta^{-1}(\mu_\be^+)$ to $X_\be$. (This is the measure
of maximal entropy for the subshift.) Let $\sigma_\be:X_\be\to X_\be$ denote the corresponding
left shift. Note that since $\be$ is Pisot, $(d_n')_1^\infty$
is eventually periodic and therefore, cannot contain unbounded strings of 0s (see, e.g., \cite{Ber}).

As is well known, any homoclinic point $\bm t$ can be obtained by projecting a point in $\mathbb Z^m$ onto
the leaf of the unstable manifold for $T_M$ passing through $\bf 0$ along the stable manifold.
In the Pisot case this implies that $T_M\bm t=\beta\bm t$, whence (\ref{eq:phit}) can be written as
\begin{equation}\label{eq:phi-pisot}
\phi_{\bm t}(\bm a)=\sum_{n\in\mathbb Z} a_n\beta^{-n}\bm t.
\end{equation}
It has been shown independently in \cite{Sch} and \cite{S01}
that $\phi_{\bm t}$ is surjective and {\em finite-to-one}, i.e., there exists $M\ge1$ such that
$\phi_{\bm t}^{-1}(\bm x)$ is at most $M$ points for any $\bm x\in\mathbb T^m$.
Furthermore, $(X_\be, \sigma_\be, \mu_\be)$ is known to be irreducible sofic in this setting (\cite{Ber}).

Recall, we define distance on a subshift $(X, \mu, \sigma)$ with the alphabet
$\A=\{0,1,\dots,k-1\}$ in the usual way as
    \[ d(\bm x, \bm y)=\sum_{i\in\mathbb Z} \frac{|x_i-y_i|}{k^{|i|}}.\]
We denote distance on $\mathbb T^m$ by $\vert x - y \vert$.
We say that a map $\phi: X \to \mathbb T^m$ is {\em $\alpha$-H\"older continuous} if
    there exists a $C>0$ such that for all $\bm x, \bm y \in X$ we have
    \[ |\phi(\bm x) - \phi(\bm y)| \leq C d(\bm x, \bm y)^\alpha. \]

\subsection{Auxiliary results}
\begin{lemma}\label{lem:holdercts}
Let $(X,\mu,\sigma)$ be an irreducible sofic subshift
endowed with the measure of maximal entropy, on the alphabet
$\A$ of cardinality~$k$. Assume we have an $\alpha$-H\"older continuous
map $\phi:X\to\mathbb T^m$ such that $\phi\sigma=T_M\phi$.

Then there exist $C, \kappa>0$ such that for any Borel set $A\subset X$ we have
\[
\H_m(\phi(A)) \leq C \mu(A)^\kappa.
\]
Here $\H_m$ is the Haar measure on $\mathbb{T}^m$ ($=$ the $m$-dimensional Lebesgue
measure restricted to $\mathbb T^m$).
\end{lemma}
\begin{proof}Without loss of generality assume $\A=\{0,1,\dots, k-1\}$.
Since $\phi$ semiconjugates $\sigma$ and $T_M$, it suffices to prove our claim for an arbitrary
cylinder $A=[w_{-n}\dots w_0.w_1\dots w_n]\subset X$. By (\ref{eq:sofic}),
we have $\mu(A)\asymp \theta^n$; we also have $\text{diam}\ A\asymp k^{-n}$.
Hence
\[
\text{diam}\ \phi(A)=O(k^{-n\alpha})=O(\mu(A)^\gamma),
\]
with $\gamma=-\alpha\log_\theta k$.
Furthermore,
\[
\H_m(\phi(A))=O(\text{diam}\ \phi(A)^m)=O((\mu(A)^\kappa),
\]
with $\kappa=m\gamma$.
\end{proof}

\begin{cor}\label{cor:proj}
The claim of Lemma~\ref{lem:holdercts} holds for $(X_\be,\mu_\be,\sigma_\be)$.
\end{cor}

\begin{lemma}\label{lem:connected}
For any cylinder $\sigma^n[w]\subset X_\be$ its image under $\phi_{\bm t}$ is path connected.
\end{lemma}
\begin{proof}
Since $T_M$ is continuous, we have that if $\phi_{\bm t}([w])$ is connected,
    then so is $\phi_{\bm t}(\sigma^k([w]))$ for all $k\in\mathbb Z$,
    in view of $\phi_{\bm t}\sigma_\be^k=T_M^k\sigma_\be$.
    So, without loss of generality, we may prove the result for $[w]$ only.

Recall that $\pi_\be(X_\be^+)=[0,1)$, a path connected set. This implies that for any
cylinders $[w]^+ = [w_1 w_2 \dots w_N]^+$ and $[w']^+ = [w_1' w_2' \dots w_N']^+$ in $X_\be^+$,
we have that there exists a chain of cylinders in $X_\be^+$, all of length~$N$, namely,
$[w^{(1)}]^+ = [w]^+,
[w^{(2)}]^+,
[w^{(3)}]^+, \dots,
[w^{(k)}]^+ = [w']^+$,
such that $\pi([w^{(j)}]^+)\cap \pi([w^{j+1}]^+)\neq\varnothing$
for $j=1,\dots, k-1$.
The key observation to see this is, for $w_n \neq 0$ that
    $w_1 w_2 \dots w_{n-1} w_n 0 0 0 .... =
     w_1 w_2 \dots w_{n-1} (w_n - 1) d_1 d_2 d_3 \dots$ and hence
    $[w_1 w_2 \dots w_{n-1} w_n]^+ \cap [w_1 w_2 \dots w_{n-1} (w_n -1)]^+ \neq \varnothing$.
Similarly
    $[w_1 w_2 \dots w_{n-1} 0]^+ \cap [w_1 w_2 \dots (w_{n-1}-1) d_1]^+ \neq \varnothing$,
    $[w_1 w_2 \dots w_{n-2} 0 0]^+ \cap [w_1 w_2 \dots (w_{n-2}-1) d_1 d_2]^+ \neq \varnothing$, etc.

Now the claim follows from (\ref{eq:phi-pisot}); indeed, for any two cylinders
$[w], [w']$ in $X_\be$, we can use
  the same chain as above (to be more precise, their two-sided analogues), and the
  images $[w^{(j)}]$ and $[w^{(j+1)}]$
  under $\phi_{\bm t}$ will intersect as well.
\end{proof}


\begin{lemma}\label{lem:nei}
For any cylinder $C\subset X_\be$ the image $\phi_{\bm t}(C)$ has non-empty interior.
\end{lemma}
\begin{proof}Let $\Xi_n$ denote the set of all cylinders $[w_{-n+1}\dots w_0.w_1\dots w_n]$
in $X_\be$. Since $\phi_{\bm t}$ is surjective, there exists $C=[w_{-n+1}\dots w_n]\in\Xi_n$ such that
$\phi_{\bm t}(C)$ has non-empty interior. Notice that
\[
\phi_{\bm t}(C)\subset \phi_{\bm t}(\underbrace{[0\dots 0.0\dots 0]}_{2n})+\sum_{-n+1}^n w_j\beta^{-j}\bm t,
\]
as the word with all 0s has no constraints in $X_\be$ as long as $2n\ge\ell$, which we assume for now.
This implies that $[0\dots 0.0\dots 0]\subset \Xi_n$ has the property in question. Hence so
does $[0\dots 0.0\dots 0]\subset \Xi_{n+\ell}$.

Now, by (\ref{eq:ell}), for any word $w'_{-n+1}\dots w'_n\in\L_{2n}(X_\be)$ we have
\[
\phi_{\bm t}([w'_{-n+1}\dots w'_n])\supset \phi_{\bm t}(\underbrace{[0\dots 0.0.\dots 0]}_{2n+2\ell})+
\sum_{-n+1}^n w'_j\beta^{-j}\bm t,
\]
which proves the claim, since any cylinder contains a longer one, hence $2n\ge\ell$ is not a real constraint.
\end{proof}

Let $\be_1=\be,\be_2,\dots, \be_m$ be the Galois conjugates of $\be$ and put
\[
\Lambda_\be = \{\bm x \in X_\be: x_i = 0\ \mathrm{for\ all}\ i \geq 1\}.
\]
Then $\phi_{\bm t}(\Lambda_\be)$ lies
in $W$, the $(m-1)$-dimensional span of the eigenvectors for $\be_2,\dots, \be_m$. Recall that $|\be_j|<1$ for
$j\in\{2,\dots,m\}$. Notice that since $\prod_1^m |\be_j|=1$, we have
\begin{equation}\label{eq:eigenv}
\min\{|\be_2|,\dots,|\be_m|,\be^{-1}\}=\be^{-1}.
\end{equation}

\begin{cor}\label{cor:nei}
The set $\phi_{\bm t}(\Lambda_\be)$ has non-empty interior in $W$.
\end{cor}

\begin{lemma}\label{lem:cube}
There exists $c=c(M)>0$ such that for any cylinder $C\in\Xi_n$ we have that
$\phi_{\bm t}(C)$ contains a cube whose sides are aligned with the axes,
with side $c\beta^{-n}$.
\end{lemma}
\begin{proof}Let $C=[w_{-n+1}\dots w_0.w_1\dots w_n]$. Consider
\[
C':=\sigma_\be^{-n-\ell+1}(C)=[x_\ell=w_{-n+1},\dots, x_{2n+\ell-1}=w_n].
\]
By (\ref{eq:ell}),
\[
\phi_{\bm t}(C')\supset \phi_{\bm t}(\Lambda)+\phi_{\bm t}([0^\ell w_{-n+1}\dots w_n]).
\]
By Corollary~\ref{cor:nei}, the first summand contains an $(m-1)$-dimensional cube
with side $c_1$, say. The second summand contains an interval which is transversal
to $\phi_{\bm t}(\Lambda)$, of length $\ge c_2\beta^{-2n}$.

This implies that $\phi_{\bm t}(C')$ contains a box with dimensions
$c_1\times\dots\times c_1\times c_2\be^{-2n}$. Now, $C=\sigma_\be^{n+\ell-1}(C')$.
The map $T_M$ contracts on $W$, with the contraction ratios $\ge\beta^{-1}$ -- see
(\ref{eq:eigenv}). On the one-dimensional eigenspace corresponding to $\be$, it expands
by $\be$. Hence follows the claim.
\end{proof}

\subsection{Toral automorphisms with small holes}
\begin{defn}Let $X$ be a compact metric space, $T:X\to X$ be a continuous invertible map
and $\mu$ be an ergodic $T$-invariant probability measure. We say that the dynamical system
$(X,\mu, T)$ possesses {\em Property S} if for any $\epsilon>0$ there exists an
open connected subset $A$ of $X$ such that
\begin{enumerate}
\item $\mu(A)<\epsilon$;
\item for all, except, possibly, a countable set of $x\in X$, there exists $n=n(x)\in\mathbb Z$,
such that $T^nx\in A$.
\end{enumerate}
\end{defn}

\begin{prop}\label{prop:Pisot}
Any Pisot toral automorphism possesses Property~S.
\end{prop}
\begin{proof}
Fix $\epsilon>0$ and choose $\C$ from Theorem~\ref{thm:main} applied to $X_\beta$. Put
$\D=\phi_{\bm t}(\C)$. Clearly, $\D$ is compact; it is also path
connected by Lemma~\ref{lem:connected}.
Furthermore, we have that $\phi_{\bm t}(\mu_\be)=\H_m$, as a unique invariant measure of maximal entropy for $T_M$.

The set $\D$ is a complete trap; indeed, for $\bm x\in\mathbb T^m$, let
$\bm a=(a_n)_{-\infty}^\infty\in\phi_{\bm t}^{-1}(\bm x)$. Since $\C$ is a complete trap,
there exists $n\in\mathbb Z$ such that $\sigma^n \bm a\in\C$. We then have
$T_M^n \bm x=\phi_{\bm t}(\sigma^n\bm a)\in\D$.

By Corollary~\ref{cor:proj}, $\H_m(\D)=O(\epsilon^\kappa)$ with some $\kappa>0$.
Finally, for $\delta>0$ put $\D'=\{\bm x\in\mathbb T^m : |\bm x-\bm y|<\delta,\forall \bm y\in\D\}$.
Since $\D$ is path connected, we have that $\D'$ is an
open (path) connected set containing $\D$ which can be made arbitrarily close to $\D$
in measure by choosing an arbitrarily small $\delta$.
\end{proof}

\begin{defn} \cite{S01} We say that $T$ is {\em generalized Pisot} if one of $\pm T, \pm T^{-1}$ is is Pisot.
\end{defn}

\begin{thm}\label{thm:pisot}
Any generalized Pisot toral automorphism has Property~S.
\end{thm}
\begin{proof}
Suppose $T=T_M$ is Pisot. Since the orbit of $T$ is the same
as of $T^{-1}$ for any $x\in\mathbb T^m$, it is clear from the definition that a complete trap
for $T$ is a complete trap for $T^{-1}$.

Not let $S=-T$ and put $\E=\D'\cup(-\D')$, where $\D'$
is constructed in the proof of Proposition~\ref{prop:Pisot}. Clearly, $\E$ is open and has
a small measure if $\D'$ does. Since $S^n=T^n$ if $n$ is even and $-T^n$ otherwise,
we have that $\E$ is a complete trap for $S$ if $\D'$ is such for $T$. To make
$\E$ connected, we connect $\D'$ and $-\D'$ by an open ``tunnel'', which
we can make as small in measure $\L_m$ as we please. Let $\E'$ denote the resulting set,
which is clearly small in measure and a complete trap for $S$.

Now, the same $\E'$ works for the case $U=-T^{-1}$, which completes the proof.
\end{proof}

\begin{cor}Any hyperbolic automorphism of $\mathbb T^2$ or $\mathbb T^3$ has Property~S.
\end{cor}
\begin{proof}Clearly, any hyperbolic $2\times2$ matrix has one eigenvalue of modulus greater
than one and one less than one. This makes it generalized Pisot.

For a $3\times 3$ hyperbolic
matrix $M$ we have that one of its eigenvalues is less than one in modulus and two greater than one
or the other way round. In either case, $T_M$ is generalized Pisot so we can apply Theorem~\ref{thm:pisot}.
\end{proof}

Consider now a general hyperbolic toral automorphism $T_M$. Arithmetic
symbolic models (similar to the one described above for the Pisot
automorphisms) have been suggested by various authors - see \cite[Section~4]{AD}
for more detail. Unfortunately, none of these appears to produce an explicit
symbolic coding space.

The following result has been proved by S.~Le Borgne:

\begin{thm}\cite{Leb} There exists an irreducible sofic shift
$(X,\nu,\sigma)$ and a surjective, H\"older continuous, finite-to-one
map $\psi:X\to\mathbb T^m$  which semiconjugates $\sigma$ and $T_M$.
\end{thm}

This result combined with Theorem~\ref{thm:main} yields that $T_M$ possesses
Property~S, except that a hole may not be connected. This is because
it is unclear whether $\psi(C)$ is connected for each cylinder $C\subset X$
(S.~Le~Borgne, private communication).

\begin{rmk}It is shown in \cite{CHS} that if we restrict our class of holes to
convex ones for the baker's map, then there exists $\delta>0$ such that for any $H$
of Lebesgue measure less than $\delta$, we have $\dim_H \J(H)>0$. It would be interesting
to establish an analogous result for hyperbolic toral automorphisms. In this
case the class of holes in question would be probably geodesically convex ones.\footnote{
A subset $E$ of a Riemmanian manifold is said to be {\em geodesically convex} if, given any two points in $E$,
there is a minimizing geodesic contained within $E$ that joins those two points.}
\end{rmk}

\subsection{Toral automorphisms with large holes}

\begin{lemma}\label{lem:ent-dim}
Let $Y$ be a subshift of $X_\be$ and let, as above, $h(Y)$ stand for
its topological entropy ($\log$ base~$\be$).
Then
\[
\dim_H(\phi_{\bm t}(Y))\ge 2h(Y).
\]
\end{lemma}
\begin{proof}Let $m$ stand for a measure of maximal entropy for $Y$. Then by the Shannon--McMillan--Breiman
theorem, for any $\varepsilon>0$, we have a set $Y'\subset Y$ of full $m$-measure and all
$\bm x\in Y'$,
\[
\be^{-2k(h(Y)+\epsilon)}\le\mu [x_{-k+1}\dots x_0. x_1\dots x_k]
 \le \be^{-2k(h(Y)-\epsilon)},\quad k\ge n(\bm x).
\]
Let $N_0$ be large enough that $\mu (S)>1/2$, where $S=\{\bm x\in Y' : n(\bm x)\le N_0\}$.

By Lemma~\ref{lem:cube}, the set $\phi_{\bm t}([x_{-k+1}\ldots x_0\cdot x_1\ldots x_k])$
contains a cube with sides $c\be^{-k}$, whence its (normalized)
$s$-Hausdorff measure $\H^s$ is $c^s\be^{-sk}$. Therefore,
\[
\H^s(\phi_{\bm t}(S)) \ge c^s \beta^{-2k(h(Y)+\epsilon)}\cdot \beta^{-ks}
=c^s \be^{-k(2h(Y)-s+\epsilon)}.
\]
Now the claim follows from the definition of the Hausdorff dimension.
\end{proof}

\begin{thm}\label{thm:large-toral}
Let $M\in GL(m,\mathbb Z)$ and $T=T_M:\mathbb T^m\to\mathbb T^m$ be the corresponding
automorphism which we assume to be generalized Pisot.

Then for any $\epsilon>0$ there exists an open connected hole
$H\subset\mathbb T^m$ such that $\H_m(H)>1-\epsilon$ and $\dim_H(\J(H))>0$.
\end{thm}
\begin{proof}Assume first $T$ to be Pisot. Then it follows from
Theorem~\ref{thm:large-holes 2} and Lemma~\ref{lem:ent-dim} that $H=\phi_{\bm t}(\G)$
has positive Hausdorff dimension. Lemma~\ref{lem:holdercts} ensures that the complement
of $H$ has measure $\H_m$ as small as we please.

Now let $S$ be generalized Pisot. For $S=T^{-1}$ we have the same survivor set for
any hole so the same $H$ does the job.

Assume $S=-T$ now. We will modify our proof of Theorem~\ref{thm:large-holes 2} in such a
way that we will get $H\subset \mathbb T^m$ with $H=-H$. Namely, put for any $\bm x\in X_\be$,
\[
\mathsf M(\bm x)=\{\bm y\in X_\be : \phi_{\bm t}(\bm y)=\pm\phi_{\bm t}(\bm x)\}.
\]
Since $\phi_{\bm t}$ is finite-to-one, there exists $M\ge1$ such that $\#\mathsf M(\bm x)\le M$
for all $\bm x\in X_\be$. Now let $Y$ be a subshift of $X_\be$ from Theorem~\ref{thm:large-holes 2};
for instance, we can take $Y=X_{\be'}$ where $1<\be'<\be$. Put
\[
\widetilde Y= \{\mathsf M(\bm x) : \bm x\in Y\}.
\]
Roughly speaking, $\widetilde Y$ consists of all elements of $Y$ together with their
-- possibly multiple -- negatives. We claim that $\widetilde Y$ is shift-invariant (hence a subshift).
Indeed, let $\bm y\in \widetilde Y$;
then we have $\phi_{\bm t}(\bm y)=\pm\phi_{\bm t}(\bm x)$ for some $\bm x\in Y$.
Thus,
\[
\phi_{\bm t}(\sigma(\bm y))=T\phi_{\bm t}(\bm y)=\pm T\phi_{\bm t}(\bm x)=
\pm \phi_{\bm t}(\sigma(\bm x)),
\]
whence $\sigma(\bm y)\in\widetilde Y$, since $\sigma(\bm x)\in Y$.

Furthermore, $h(\widetilde Y)=h(Y)$,
since we only add at most $M$ sequences for each $\bm x\in Y$. By our construction,
$-\phi_{\bm t}(\widetilde Y)=\phi_{\bm t}(\widetilde Y)$.
Following the proof of Theorem~\ref{thm:large-holes 2}, we set
$\widetilde \Sigma_n:=\Sigma_n\setminus \bigl\{[w] : w\in \L(\widetilde Y)\bigr\}$.

As before, we have that $\mu(\widetilde \Sigma_n) \to 1$ as $n \to \infty$.
Moreover we see that $\widetilde Y \subset \J(\widetilde \Sigma_n)$ and
    $h(\J(\widetilde \Sigma_n) > 0$.

Note that if $K$ is even then $S^K = T^K$.
To get connectedness, we again let $[w^*] \in \widetilde \Sigma_n$ be
    such that $\mu[w^*]$ is minimized and let
\[ \G_n = (\widetilde \Sigma_n \setminus \{[w^*]\}) \cup \{\sigma^{-K}([w^*])\} \]
with $K$ sufficiently large and even.
The rest of the proof follows as before.

The case where $S = -T^{-1}$ follows in exactly the same way because
    the orbit of $-T$ and $-T^{-1}$ are the same.
\end{proof}

\begin{rmk}\label{rem:final}
Recall the definition of the escape rate for an invertible map $T:X\to X$ with an invariant measure $m$ and a hole $H$.
Put
\[
e(x)=e_H(x) = \min \{n\ge 0 : T^{|n|} x\in H\}
\]
and $E_n = \{x \in X : e(x) > n\}$. Now, put $\delta = \delta (T, H, m) = \lim_{n\to\infty} m(E_n)^{1/n}$. The quantity
\[
\mathfrak e(T, H, m)=-\log \delta (T,H, m)
\] 
is usually referred to as the {\em escape rate}. Theorem~\ref{thm:pisot}
says that there exist arbitrarily small connected holes $H$ such that $\mathfrak e(T, H, m)=+\infty$
for a wide class of toral automorphisms $T$. 

To our best knowledge, this effect has not been studied in the literature (see Section~\ref{sec:intro}). 
\end{rmk}

\end{document}